\documentclass[11pt]{article}

\usepackage{mathrsfs,amssymb,amsmath,amsfonts}
\usepackage{amsthm}
\usepackage{bbm}
\usepackage{bm}
\usepackage{tipa}
\usepackage{enumerate}
\usepackage{color}
\usepackage[titletoc,title]{appendix}
\usepackage{mathtools}
\usepackage{cite}
\usepackage[hyphens]{url}
\usepackage[bookmarks=false,hidelinks]{hyperref}

\usepackage{arydshln}
\usepackage{dsfont}
\usepackage{upgreek}

\newtheorem{theorem}{Theorem}[section]

\newtheorem{proposition}[theorem]{Proposition}

\theoremstyle{definition}

\newenvironment{keywords}{{\bf Key words: }}{}

\newcommand{\abs}[1]{\left\vert#1\right\vert}
\newcommand{\norm}[1]{\left\Vert#1\right\Vert}

\newcommand{\T}{\mathcal {T}}

\newcommand{\R}{\mathbb{R}}
\newcommand{\Pp}{\mathbb{P}}
\newcommand{\E}{\mathbb{E}}

\newcommand{\expectation}[3][0]{%
   \ifcase#1
   \mathbb{E}[ #2 \mid #3 ]
      \or \mathbb{E} \bigl[ #2 \bigm\vert #3 \bigr]
      \or \mathbb{E} \Bigl[ #2 \Bigm\vert #3 \Bigr]
      \or \mathbb{E} \biggl[ #2 \biggm\vert #3 \biggr]
      \or \mathbb{E} \Biggl[ #2 \Biggm\vert #3 \Biggr]
   \else
   \mathbb{E} \left[ #2  \;\middle\vert\; #3 \right)]
   \fi}

\newcommand{\udots}{\mathinner{\mskip1mu\raise1pt\vbox{\kern7pt\hbox{.}}
\mskip2mu\raise4pt\hbox{.}\mskip2mu\raise7pt\hbox{.}\mskip1mu}}

\title{ {\bf A probabilistic representation for heat flow of harmonic map on manifolds with time-dependent Riemannian metric} }
\author{
{\bf    Xin Chen$^{a)}$, Wenjie Ye$^{a)}$}
\thanks{E-mail address: chenxin217@sjtu.edu.cn (X. Chen), yewenjie@sjtu.edu.cn(W.J. Ye)}
\\
\footnotesize {$^{a)}$ School of Mathematical Sciences, Shanghai Jiaotong University, Shanghai 200240, China}
}

\begin{document}

\maketitle
\numberwithin{equation}{section}
\begin{abstract}
    In this paper we will give a probabilistic representation for the heat flow of harmonic map with time-dependent Riemannian metric
   via a forward-backward stochastic differential equation on Riemannian manifolds. Moreover, we can provide
   an alternative stochastic method for the proof of existence of a unique
   local solution for heat flow of harmonic map with time-dependent Riemannian metric.
\end{abstract}

\\

\section{Introduction}

 Eells and Sampson firstly studied the harmonic map between two different Riemannian manifolds
in \cite{ES}. There are various deep applications of harmonic map in the area of differential geometry and topology, we refer
readers to the monograph \cite{LW} for more details. In \cite{ES}, Eells and Sampson also introduced the heat flow of harmonic
map, which could be viewed as a parabolic version of harmonic map and has been applied to prove the
existence of harmonic map for some special Riemannian manifold. Moreover, different properties of such heat flow of harmonic map
have been systematically investigated by \cite{CS,S1,S2}, see also \cite{LW} and references therein.

 In \cite{Ke1}, Kendall proved that the harmonic map from (a Riemannian manifold) $M$ to
 (a Riemannian manifold) $N$ was characterized by a map from $M$ to $N$ which mapped an $M$-valued Brownian motion
 to an $N$-valued martingale associated with the Levi-Civita connection on $N$. Then it is natural to use stochastic methods to study
 the subjects on heat flow of harmonic map. Thalmaier \cite{T} proved several properties concerning about the singularity for heat flow
of harmonic map by using some ideas of backward stochastic differential equation on a manifold. Guo, Philipowski and Thalmaier
\cite{GPT} characterized the differential for heat flow of harmonic map with time-dependent Riemannian metric
through martingale arguments, based on which several
gradient estimates were obtained by stochastic methods.

In this paper, we will give a probabilistic representation for heat flow of harmonic map with time-dependent Riemannian metric
(not the representation for its differential as in \cite{GPT}). Moreover, based on this we will provide a
proof of the existence of a unique local solution for  heat flow of harmonic map with time-dependent Riemannian metric. See Theorem
\ref{t2-1} below for more details. We also refer readers to \cite{LY,Mu} and references therein for more background knowledge and
applications about the heat flow of harmonic map with time-dependent Riemannian metric in the area of geometric analysis.

In fact, as explained in Theorem \ref{t2-1}, we will apply a forward-backward stochastic differential equation
(which will be written as FBSDE from now for simplicity)
to give a representation for the heat flow of harmonic map with time-dependent Riemannian metric, where the forward equation
is a stochastic differential equation (which will be written as SDE from now) associated with $g_t$-Brownian motion
on based Riemannian  manifold $M$, and the backward equation is an $N$-valued backward stochastic differential equation
(which will be written as BSDE from now) introduced
in Chen and Ye \cite{chen_StudyBackwardStochastic_2020} (for target manifold $N$), whose solution is not necessarily restricted in
only one local coordinate.

The linear version of BSDE was firstly introduced by Bismut\cite{bismut_LinearQuadraticOptimal_1976}. Pardoux and Peng \cite{pardoux_AdaptedSolutionBackward_1990} made a break through to prove
the existence of a unique solution for general (non-linear) BSDE under uniformly Lipschitz continuity conditions on the generators.
Moreover, Peng \cite{peng_ProbabilisticInterpretationSystems_1991} and Pardoux and Peng\cite{pardoux_BackwardStochasticDifferential_1992}
established an equivalence between FBSDE and quasi-linear parabolic partial differential systems.  We also refer readers to
Ma, Protter and Yong\cite{MPY},  Wu and Yu\cite{WY} (and references therein) for more details on the subjects about
the connection between FBSDE and PDE.

We also give some remarks on our results.
\begin{itemize}
\item [(1)] As explained in \cite{chen_StudyBackwardStochastic_2020}, the backward equation in \eqref{Gamma-bsde}
could be viewed a BSDE with quadratic growth generators in ambient Euclidean space $\R^{L_2}$ for target manifold
$N$. By our knowledge, among all the known results on the existence of a unique solution for FBSDE with quadratic growth
generators, such as Kupper, Luo and Tangpi\cite{kupper_MultidimensionalMarkovianFBSDEs_2019}, the uniform boundedness of Malliavin derivatives
for the terminal value is required, which implies that the forward equation is a SDE with additive noise. But in our setting, the forward
equation in \eqref{Gamma-bsde} is usually a SDE with multiplicative noise, hence all the known results could not be applied directly.
In this paper
we will develop a new iteration method to solve this kind of FBSDE. Although \eqref{Gamma-bsde} is a decoupled FBSDE
(i.e. the forward SDE does not depend on $Y$ or $Z$ in the backward equation),
we believe that our method could also be applied to solve a
coupled FBSDE.
Compared with \eqref{Gamma-bsde}, some extra non-trivial calculations for the coupled terms are required
for a coupled FBSDE. Due to the limitation for the length of this paper,
we will address such problem in a separate paper, which will extend the results of  Kupper, Luo and Tangpi\cite{kupper_MultidimensionalMarkovianFBSDEs_2019}
to the case with multiplicative noise.

\item [(2)]  By our knowledge, all existing results on the connection between FBSDE and partial differential system, such as
\cite{peng_ProbabilisticInterpretationSystems_1991,pardoux_BackwardStochasticDifferential_1992,MPY,WY},
are only for the equations in Euclidean space.
Different from the results mentioned above, in this paper we will give an expression for the solution of a PDE in a Riemannian manifold
by a FBSDE taking values in Riemannian manifolds (see e.g. \eqref{Gamma-bsde} below).
This gives us another possible
way to study more different non-linear PDEs in a Riemannian manifold.

\item [(3)] In \cite{F,RWZZ}, the stochastic heat flow of harmonic map has been investigated, which was applied to
study the diffusive motion of loops on a Riemannian manifold. Our formulation in this paper will provide
a natural way to introduce the backward stochastic heat flow of harmonic map, which is still unknown by our knowledge.
In fact, using the procedures in \cite{MYZ}, if we choose a non-Markovian terminal value in \eqref{Gamma-bsde} (i.e.
replacing the non-random $h:M\to N$ by an $N$-valued $\mathscr{F}_T^t$ measurable random map $\hat h:M\times \Omega \to N$, where
$\mathscr{F}_T^t:=\sigma\{W_s-W_T;0\le t\le s\le T\}$), then the solution of \eqref{Gamma-bsde} is related
to an $N$-valued backward stochastic partial differential equation. We believe that such equation could be viewed
as a backward stochastic heat flow of harmonic map and plan to study it in the future.
\end{itemize}

The rest of this paper is organized as follows. In Section \ref{section-2} we will present a brief
introduction on some preliminary knowledge and main result of this paper. In Section \ref{section-3} the proof
for main theorem of this paper will be given. In Section \ref{section-4} we will prove Proposition \ref{p4-1}.

\section{Preliminaries and Main Results}\label{section-2}

  Let $T>0$ be a fixed constant. Suppose $M$ is an $m$-dimensional compact manifold endowed with a time-dependent Riemannian metric
$\{g_t\}_{t\in [0,T]}$, i.e. $(M,g_t)$ is a compact Riemannian manifold for every fixed $t\in [0,T]$.
We also assume that
\begin{equation}\label{e2-0}
\mathcal{R}:=\sup_{x\in M, t\in [0,T]}
\left|\frac{\partial g_t(x)}{\partial t}+{\rm Ric}_{g_t}(x)\right|_{\left(T_xM \times T_x M, g_t\times g_t\right)}<\infty,
\end{equation}
where ${\rm Ric}_{g_t}$ denotes the Ricci curvature tensor on $M$ associated with $g_t$, $\left|\cdot\right|_{\left(T_xM \times T_x M, g_t\times g_t\right)}$
is the norm on $T_x M \times T_x M$ induced by product Riemannian metric $g_t\times g_t$. Note that here we also make the
assumption that the constant $\mathcal{R}$ in \eqref{e2-0} is independent of $T$.
Suppose $(N,\mathbf h)$ is an $n$-dimensional compact Riemannian manifold (endowed with a Riemannian metric $\mathbf h$ on $N$).
{\em Denote the covariant derivative on $M$ with respect to Levi-Civita connection associated with $g_t$ by $\nabla^{g_t}$ and
denote the covariant derivative on
$N$ with respect to Levi-Civita connection associated with $\mathbf h$ by $\nabla^{\mathbf h}$ respectively.}
According to Nash embedding theorem, given a $t\in [0,T]$, there exist smooth isometric embedding $\Phi_t:(M,g_t)\to \R^{L_1}$ and
$\Psi: (N,\mathbf h)\to \R^{L_2}$. {\em Without cause of confusion, we will
use $\bar \nabla$ to represent the (covariant) derivative on both $\R^{L_1}$ and $\R^{L_2}$.}
For any $k_1,k_2\in \mathbb{N}$, let
\begin{equation*}
\begin{split}
C^{k_1,k_2}_b([0,T]\times \R^{L_1};\R^{L_2})
:=\Big\{& u:[0,T]\times \R^{L_1}\to \R^{L_2};
\sum_{i=1}^{k_1}\sup_{t\in [0,T];x\in \R^{L_1}}\left|\frac{\partial^i u(t,x)}{\partial t^i}\right|<\infty,\\
&\sum_{i=1}^{k_2}\sup_{t\in [0,T];x\in \R^{L_1}}\left|\bar \nabla^{i} u(t,x)\right|<\infty\Big\},
\end{split}
\end{equation*}
\begin{equation*}
\begin{split}
C_b^{k_1,k_2}([0,T]\times M;N):=\Big\{& u:[0,T]\times M\to N\subset \R^{L_2};
\sum_{i=1}^{k_1}\sup_{t\in [0,T];x\in M}\left|\frac{\partial^i u(t,x)}{\partial t^i}\right|<\infty,\\
&\sum_{i=1}^{k_2}\sup_{t\in [0,T];x\in M}\left|\nabla^{i,g_t}u(t,x)\right|<\infty\Big\},
\end{split}
\end{equation*}
where $\bar \nabla^i$ and $\nabla^{i,g_t}$ denote the $i$-th covariant derivative with respect to
$\bar \nabla$ and $\nabla^{g_t}$ respectively.

Let $\{e_i\}_{i=1}^{L_1}$ be the standard orthonormal basis of $\R^{L_1}$, we define
$A_i(t,x):=\nabla^{g_t}\left\langle \Phi_t(x),e_i\right\rangle_{\R^{L_1}}$, $x\in M$, $1\le i \le L_1$, where
$\langle\cdot, \cdot \rangle_{\R^{L_1}}$ denotes the standard Euclidean inner product in $\R^{L_1}$. It is easy to see
that $A_i(t,x)\in T_x M$ for every $x\in M$ and $A_i(t,\cdot)$ is a smooth vector fields on $M$.
Since $\Phi_t:M\to \R^{L_1}$ is an isometric embedding, we know that $A_i(t,x)=\left\langle \nabla^{g_t}\Phi_t(x),e_i\right\rangle_{\R^{L_1}}$
$=\Pi_{M}^t(x) e_i$ for every $x\in M$ and $t\in [0,T]$, where $\Pi_M^t(x):\R^{L_1}\to T_x M$ represents the orthogonal projection from
$\R^{L_1}$ to $T_x M$ with respect to $g_t$. Given a $t\in [0,T]$ and an $x\in M$, set $A(t,x):\R^{L_1}\to T_x M$ by
$A(t,x)\xi:=\sum_{i=1}^{L_1}\xi_i A_i(t,x)$, therefore we have $A(t,x)=\Pi_M^t(x)$.

For arbitrarily given smooth vector fields $X,Y$ on $N$, let $\bar X$, $\bar Y$ be the (smooth) extension of $X$, $Y$
on $\R^{L_2}$ (which satisfies that $\bar X(p)=X(p)$, $\bar Y(p)=Y(p)$ for any $p\in N$), then we have
$\nabla_X Y(p)=\Pi_N(p)\left(\bar \nabla_{\bar X}\bar Y(p)\right)$, where
$\Pi_N(p):\R^{L_2}\to T_p N$ is the orthogonal projection from
$\R^{L_2}$ to $T_p N$ with respect to $\mathbf h$.
And we define the second fundamental form $\Gamma(p):T_pN \times T_p N \to T_p^{\bot} N$ at $p\in N$ as follows
\begin{equation}\label{e2-2}
\begin{split}
\Gamma(p)(u,v):&=\bar \nabla_{\bar X}\bar Y(p)-\nabla_X Y(p)\\
&=\bar \nabla_{\bar X}\bar Y(p)-\Pi_N(p)\left(\bar \nabla_{\bar X}\bar Y(p)\right), \ \ \forall\ 
u,v\in T_p N,
\end{split}
\end{equation}
where $X$, $Y$ are arbitrarily given smooth vector fields on $N$ such that $X(p)=u$, $Y(p)=v$,
$\bar X$, $\bar Y$ are any smooth
extension
of $X$ and $Y$ to $\R^{L_2}$,
and $T_p^{\bot} N$ denotes the orthogonal complement of $T_p N$ in $\R^{L_2}$.
The value of
$\Gamma(p)(u,v)$ is independent of the choice of  $X$, $Y$, $\bar X$, $\bar Y$.

Let $\operatorname{dist}_N(p)$ be the distance between
$p\in \R^{L_2}$
and $N$ such that
\[
\operatorname{dist}_N(p):=\inf\{\abs{p-q};q\in N\subset \mathbb{R}^{L_2}\},\ \ \forall\ p\in \R^{L_2}
\]
where $\abs{p-q}$ denotes the Euclidean distance between $p$ and $q$ in $\mathbb{R}^{L_2}$. For any $r>0$, we set $B(N,r):=\{ p\in \mathbb{R}^{L_2};\operatorname{dist}_N(p)<r \}$. Note that $N$ could be viewed as a compact sub-manifold of $\mathbb{R}^{L_2}$, and
there exists a $\delta_N> 0$ such that the nearest projection map $P_N:B(N,3\delta_N)\rightarrow N$ 
and the square of distance function $\operatorname{dist}_N^2: B(N,3\delta_N)\rightarrow \R_+$ are smooth.
Here for every $p\in B(N,3\delta_N)$
,
$P_N(p)=q$ so that $q\in N$ is the unique element in $N$ satisfying $\abs{p-q} = \operatorname{dist}_N(p)$.
We also have the following expression for
the second fundamental form $\Gamma(p)$,
\begin{equation*}
\Gamma(p)(u,u)=\sum^{L_2}_{i,j=1}\frac{\partial^2 P_N}{\partial p_i\partial p_j}(p)u_iu_j,\ p\in N, u=(u_1,\cdots,u_{L_2})\in T_pN\subset \R^{L_2}.
\end{equation*}
Then we can define a smooth extension of $\Gamma$ to $\bar \Gamma:\R^{L_2} \to L(\R^{L_2}\times \R^{L_2}
;\R^{L_2})$ as follows, where
$L(\R^{L_2}\times \R^{L_2};\R^{L_2})$ denotes the collection of all linear maps from $\R^{L_2}\times \R^{L_2}$ to $\R^{L_2}$.
\begin{equation}\label{e2-3}
    \bar \Gamma(p)(u,u):=\begin{cases}
        \phi_N(\operatorname{dist}_N(p))\sum^{L_2}_{i,j=1}\frac{\partial^2 P_N}{\partial p_i\partial p_j}
        (P_N(p))u_i u_j,\quad p\in B(N,2\delta_N),\ u=(u_1,\cdots,u_{L_2})\in \R^{L_2}, \\
        0,\quad p\in\R^{L_2}
        /B(N,2\delta_N),\ u=(u_1,\cdots,u_{L_2})\in \R^{L_2}.\\
    \end{cases}
\end{equation}
Here $\phi_N\in C^\infty(\mathbb{R};\mathbb{R})$ is a cut-off function such that
\[
    \phi_N(s)=\begin{cases}
        1, s< \delta_N,\\
        \in (0,1),\ s\in[\delta_N,2\delta_N],\\
        0,\ s>2\delta_N.\\
    \end{cases}
\]
By \eqref{e2-3} it is easy to see that $\bar \Gamma(p)(u,u)=\Gamma(p)(u,u)$ for every $p\in N\subset \R^{L_2}$ and
$u\in T_p N\subset \R^{L_2}$.
We refer readers to \cite[Section III.6]{Cha}, \cite[Chapetr 6]{D} or \cite[Section 1.3]{LW} for detailed introduction
 concerning about various properties for sub-manifold of an ambient Euclidean space, see also \cite{chen_StudyBackwardStochastic_2020}
 for the construction of $\bar \Gamma$ as an extension of $\Gamma$.

Through this paper, we will fix a complete probability space $(\Omega,\mathscr F,\mathbb P)$ and a standard
$L_1$-dimensional  Brownian motion $\{W_s=(W_s^1,W_s^2,\cdots,W_s^{L_1})\}_{s\in [0,T]}$ on $(\Omega,\mathscr F,\mathbb P)$.
Now we consider the SDE on $M$ as follows,
\begin{equation}\label{e2-1}
\begin{cases}
& dX_s^{t,x}=\sum_{i=1}^{L_1}A_i(s,X_s^{t,x})\circ dW_s^i,\ \
0\le t\le s\le T,\\
& X_t^{t,x}=x,\ \ \ x\in M,
\end{cases}
\end{equation}
where $\circ dW_s^i$ denotes the Stratonovich integral with respect to $W_s^i$.

As explained in \cite[Exanple 1]{EL1} (see more details in \cite[Chapter 1]{EL}), it holds that
\begin{equation}\label{e2-1a}
\sum_{i=1}^{L_1}A_i(t,\cdot)A_i(t,\cdot)f=\Delta_{g_t}f,\ \ \forall\ f\in C^\infty(M),
\end{equation}
where $\Delta_{g_t}$ is the Laplace-Beltrami operator
associated with $g_t$. Hence the infinitesimal generator for time-inhomogeneous Markovian process
$\{X_s^{t,x}\}_{s\in [t,T],x\in M}$ is $\Delta_{g_s}$, which implies immediately that the distribution of
$X_{\cdot}^{t,x}$ is the same as that of $M$-valued Brownian motion associated with time-changing metric
$\{g_s;s\in [t,T]\}$ introduced by \cite{ACT,C}.

Moreover, by the same way of \eqref{e2-3}, we could extend $A_i(s,\cdot)$ above to a (smooth) vector fields
$\bar A_i(s,\cdot):\R^{L_1}\to \R^{L_1}$ on $\R^{L_1}$. With the extended vector fields $\bar A_i(s,\cdot)$, we consider
the following SDE on $\R^{L_1}$,
\begin{equation}\label{e2-4}
\begin{cases}
& d\bar X_s^{t,x}=\sum_{i=1}^{L_1}\bar A_i(s,\bar X_s^{t,x})\circ dW_s^i,
\ \ 0\le t\le s\le T,\\
& \bar X_t^{t,x}=x,\ \ \ x\in \R^{L_1}.
\end{cases}
\end{equation}
Although the coefficients $\bar A_i$, $1\le i \le L_1$ for \eqref{e2-4} (and the coefficients $A_i$, $1\le i \le L_1$ for \eqref{e2-1})
are time-dependent, repeating the proof of \cite[Proposition 1.2.8]{hsu2002StochasticAnalysisManifoldsa}
(note that we can still apply the inequality (1.2.4) in the proof of \cite[Proposition 1.2.8]{hsu2002StochasticAnalysisManifoldsa} for time-dependent vector fields) we can obtain that
$\bar X_s^{t,x}=X_s^{t,x}$ for every $x\in M\subset \R^{L_1}$ and $0\le t\le s \le T$, where
$X_s^{t,x}$ is the solution of $M$-valued SDE \eqref{e2-1}.

We write the components of $\bar A_i(t,x)$ as $\bar A_i(t,x)=\left(\bar A_{i1}(t,x),\cdots,\bar A_{iL_1}(t,x)\right)$ and
for every $u\in C_b^{1,\infty}([0,T]\times \R^{L_1};\R^{L_2})$,
$x\in \R^{L_1},\ 1\le i \le L_1$,
we set
\begin{equation}\label{t2-1-2a}
\begin{split}
\left(\bar A(t,x)\bar \nabla u(t,x)\right)_i:&=\sum_{j=1}^{L_1}\bar A_{ij}(t,x)\frac{\partial u(t,x)}{\partial x_j}
=\langle \bar A_i, \bar \nabla u(t,x)\rangle_{\R^{L_1}}\in \R^{L_2},
\\
\bar A(t,x)\bar \nabla u(t,x):&=\left(\left(\bar A(t,x)\bar \nabla u(t,x)\right)_1,\cdots, \left(\bar A(t,x)\bar \nabla u(t,x)\right)_{L_1}
\right)\\
&=\sum_{i=1}^{L_1}e_i \otimes\langle \bar A_i(t,x), \bar \nabla u(t,x)\rangle_{\R^{L_1}}\in
\R^{L_1L_2}
\end{split}
\end{equation}
(recall that $\bar \nabla$ denotes the gradient operator
in $\R^{L_1}$). Moreover, note that $A_i(t,x)=\Pi_M^t(x)e_i$ and
$\bar A_i(t,\cdot):\R^{L_1}\to \R^{L_1}$ is an extension of $A_i(t,\cdot)$ as explained before, from \eqref{t2-1-2a} we have
\begin{equation}\label{t2-1-3a}
\begin{split}
\bar A(t,x)\bar \nabla u(t,x)&=
\sum_{i=1}^{L_1}e_i\otimes\left\langle A_i(t,x), \bar \nabla u(t,x)\right\rangle_{\R^{L_1}}
=\sum_{i=1}^{L_1}e_i\otimes\left\langle \Pi_M^t(x)e_i, \bar \nabla u(t,x)\right\rangle_{\R^{L_1}}\\
&=\sum_{i=1}^{L_1}e_i\otimes\left\langle e_i, \Pi_M^t(x)\left(\bar \nabla u(t,x)\right)\right\rangle_{\R^{L_1}}
=\sum_{i=1}^{L_1}e_i\otimes\left\langle e_i, \nabla^{g_t}u(t,x)\right\rangle_{\R^{L_1}}\\
&=\nabla^{g_t}u(t,x),\ \forall\ x\in M \subset \R^{L_1},
\end{split}
\end{equation}
where in the fourth equality above we have used the fact that
$\Pi_M^t(x)\left(\bar \nabla u(t,x)\right)=\nabla^{g_t}u(t,x)$ for every $x\in M$.
Furthermore, from now on we use the notation as follows
\begin{equation*}
\begin{split}
&\quad \quad\bar \Gamma(p)\left(\bar A(t,x) \bar \nabla u(t,x), \bar A(t,x) \bar \nabla u(t,x)\right)\\
&:=\sum_{i=1}^{L_1}\bar \Gamma(p)\left(\left(\bar A(t,x) \bar \nabla u(t,x)\right)_i,
\left(\bar A(t,x) \bar \nabla u(t,x)\right)_i\right)\\
&=\bar \Gamma(p)\Big({\rm tr}\left(\bar A(t,x) \bar \nabla u(t,x)\otimes \bar A(t,x) \bar \nabla u(t,x)\right)\Big),
\ \ x\in \R^{L_1},\ p\in \R^{L_2}.
\end{split}
\end{equation*}
Here ${\rm tr}\left(\bar A(t,x) \bar \nabla u(t,x)\otimes \bar A(t,x) \bar \nabla u(t,x)\right)$ denotes the trace with respect to
components in $\R^{L_1}\times \R^{L_1}$
(endowed with standard Euclidean metric) for
tensor product $\bar A(t,x) \bar \nabla u(t,x)\otimes \bar A(t,x) \bar \nabla u(t,x)\in
\R^{L_1L_2}\times \R^{L_1L_2}$.
For any $f\in C^{2}(M;N)$, by \cite[Section 1.2 and Section 1.3]{LW} we have
\begin{equation}\label{e2-5}
\begin{split}
&-\Delta_{g_t} f(x)+
\Gamma(f(x))\left(\nabla^{g_t}f(x),
\nabla^{g_t}f(x)\right)
=-\Delta_{g_t,\mathbf{h}} f(x),\ \forall\ x\in M,
\end{split}
\end{equation}
where $\Delta_{g_t} f(x)$ means the Laplace-Beltrami operator $\Delta_{g_t}$ acts on $\R^{L_2}$-valued function
$f\in C^{2}(M;\R^{L_2})$ (since
$N \subset \R^{L_2}$), $\Delta_{g_t,\mathbf{h}}f:={\rm tr}_{g_t}\left(\nabla^{\mathbf h} df\right)$ denotes the tension field of the $C^{2}$ map
$f:M \to N$ with respect to Riemannian metric $g_t$ on $M$ and $\mathbf h$ on $N$.

Then for a fixed smooth map $h\in C^\infty(M;N)$, we introduce a forward-backward stochastic
different equation on
$M \times \R^{L_2}$ (in particular with solution
$\left(X_s^{t,x},Y_s^{t,x}, Z_s^{t,x}\right)\in M\times \R^{L_2}\times \R^{L_1L_2}$
for every
$0\le t\le s\le T$,
$x\in M$),
 \begin{equation}\label{Gamma-bsde}
    \left\{\begin{aligned}
   & dX_s^{t,x}=\sum_{i=1}^{L_1} A_i(s,X_s^{t,x})\circ dW_s^i,\\
    &Y^{t,x}_s=\bar h(X^{t,x}_T) -\sum_{i=1}^{L_1}\int^T_s\frac{1}{2}\bar \Gamma(Y^{t,x}_r)\left(Z^{t,x,i}_r,Z^{t,x,i}_r\right) \,dr
    -\sum_{i=1}^{L_1}\int^T_sZ^{t,x,i}_r\,dW^i_r,\\
    & X_t^{t,x}=x,\ \
    0\le t\le s\le T.
    \end{aligned}\right.
\end{equation}
Here $\bar h:\R^{L_1}\to \R^{L_2}$ is the smooth extension of $h:M\to N$ defined by the same way
as \eqref{e2-3}, $\bar \Gamma(p):\R^{L_2}\times \R^{L_2}\to \R^{L_2}$ is defined by \eqref{e2-3} and
we use the notation $Z_s^{t,x}=\left(Z_s^{t,x,1},\cdots, Z_s^{t,x,L_1}\right)$ with
$Z_s^{t,x,i}\in \R^{L_2}$ for every $1\le i \le L_1$.

For every $q\ge 1$, $0\le T_1<T_2$ and $L,L_1,L_2\in \mathbb{N}_+$, let
\begin{equation*}
       \begin{aligned}
            \mathscr S^q([T_1,T_2];\mathbb R^L):=&\Big\{ Y:[T_1,T_2]\times\Omega\rightarrow\mathbb{R}^L; Y\  \text{is predictable},\mathbb{E}\Big[\sup_{t\in[T_1,T_2]}\abs{Y_t}^q\Big]<\infty,\\
            &t\mapsto Y_t(\omega)\   \text{is continuous on}\  [T_1,T_2] \ \text{for}\  a.s.\  \omega\in \Omega \Big\};
       \end{aligned}
\end{equation*}
\begin{equation*}
       \begin{aligned}
            \mathscr M^q([T_1,T_2];\mathbb R^L)
            :=\Big\{ Z:[T_1,T_2]\times\Omega\rightarrow\mathbb{R}^{L}; Z\  \text{is predictable},\mathbb{E}\Big[\Big(\int^{T_2}_{T_1}\abs{Z_t}^2\,dt\Big)^{q/2}\Big]<\infty   \Big\};
       \end{aligned}
   \end{equation*}
\begin{equation*}
    \begin{aligned}
    \mathscr S^q\oplus\mathscr M^q\left([T_1,T_2];\R^{L_1}\times \R^{L_2}\right)
    :=&
    \Big\{(X,Y,Z); X\in \mathscr S^q([T_1,T_2];\R^{L_1}),\\
    &\quad Y\in \mathscr S^q([T_1,T_2];\R^{L_2}),\
    Z\in \mathscr M^q\left([T_1,T_2];\mathbb R^{L_1L_2}\right)\Big\}.
    \end{aligned}
\end{equation*}

Now we will provide the main theorem of this paper which concerns about a probabilistic representation
of the solution for heat flow of harmonic map with respect to time-dependent Riemannian metric.
\begin{theorem}\label{t2-1}
There exists a $T_0>0$ such that for every $x\in M$, we could find
a unique solution $\left(X_{\cdot}^{t,x},Y_{\cdot}^{t,x},Z_{\cdot}^{t,x}\right)\in
\mathscr S^2\oplus\mathscr M^2\left([t,T_0];\R^{L_1}\times \R^{L_2}\right)$
of \eqref{Gamma-bsde} in time
interval $0\le t \le s \le T_0$ which satisfies that
\begin{equation}\label{t2-1-0a}
|Z_s^{t,x}(\omega)|\le C_1,\ ds\times\Pp-{\rm a.e.}\ (s,\omega)\in [t,T_0]\times \Omega,\ \forall\ x\in M,\ t\in [0,T_0],
\end{equation}
for some $C_1>0$.
At the same time we have  $Y_s^{t,x}\in N$ a.s. for all $0\le t\le s\le T_0$
and $x\in M$.

Moreover, let $v(t,x):=Y_{t}^{t,x}$ for every $t\in [0,T_0]$, $x\in M$, then
$v\in C_b^{1,\infty}([0,T_0]\times M; N)$ (which is non-random)
is the unique classical solution to the heat flow of harmonic map with respect to
time-dependent Riemannian metric as follows
\begin{equation}\label{t2-1-1}
\begin{cases}
& \frac{\partial v(t,x)}{\partial t}=-\frac{1}{2}\Delta_{g_t, \mathbf h}v(t,x), \ t\in [0,T_0],\\
& v(T_0,x)=h(x).
\end{cases}
\end{equation}
\end{theorem}

\section{Proof of the Theorem \ref{t2-1}}\label{section-3}
\begin{proof}[Proof of Theorem \ref{t2-1}]
Through the proof, all the constants $c_i$ are chosen to be independent of $T$.

{\bf Step (i)} Given a time horizon $T>0$ and a $u\in  C^{1,\infty}_b\left([0,T]\times \R^{L_1};\R^{L_2}\right)$,
let $\left(\bar X_\cdot^{t,x}, Y_\cdot^{t,u,x}, Z_\cdot^{t,u,x}\right)\in
\mathscr S^2\oplus\mathscr M^2\left([t,T];\R^{L_1}\times \R^{L_2}\right)$ (we omit the parameter
$T$ here for simplicity of notation)
be the unique solution to the following FBSDE,
\begin{equation}\label{t2-1-2}
\begin{cases}
&\bar X_s^{t,x}=x+\sum_{i=1}^{L_1}\int_t^s \bar A_i(r,\bar X_r^{t,x})\circ dW_r^i\\
&\quad\quad=x+\sum_{i=1}^{L_1}\int_t^s\bar A_i(r,\bar X_r^{t,x})dW_r^i+\sum_{i=1}^{L_1}\frac{1}{2}\int_t^s\langle \bar \nabla \bar A_i
(r,\bar X_r^{t,x}),  \bar A_i(r,\bar X_r^{t,x})\rangle_{\R^{L_1}} dr,
\\
    &Y^{t,u,x}_s=\bar h(\bar X^{t,x}_T) -\int^T_s\frac{1}{2}\bar \Gamma(Y^{t,u,x}_r)
    \Big(\bar A(r,\bar X_r^{t,x})
    \bar \nabla u(r,\bar X_r^{t,x}),\bar A(r,\bar X_r^{t,x})\bar \nabla u(r,\bar X_r^{t,x})\Big) \,dr\\
    &\quad\quad\quad\quad-\sum_{i=1}^{L_1}\int^T_sZ^{t,u,x,i}_r\,dW^i_r,\ \ \forall\
    0\le t\le s\le T.
\end{cases}
\end{equation}
Here we use the notation
$$\left\langle \bar \nabla \bar A_i(r,x),  \bar A_i(r,x)\right\rangle_{\R^{L_1}}:=
\left(\sum_{k=1}^{L_1}\frac{\partial \bar A_{i1}(r,x)}{\partial x_k}\bar A_{ik}(r,x),\cdots,\sum_{k=1}^{L_1}\frac{\partial \bar A_{iL_1}(r,x)}{\partial x_k}\bar A_{ik}(r,x)\right)\in \R^{L_1},\ 1\le i \le L_1,$$
and
we write the components of $Z_s^{t,u,x}$ as
$Z_s^{t,u,x}=\left(Z_s^{t,u,x,1},\cdots, Z_s^{t,u,x,L_1}\right)$.

It is easy to verify that  the coefficients of \eqref{t2-1-2} are uniformly Lipschitz continuous, by \cite[Theorem 3.1]{pardoux_AdaptedSolutionBackward_1990}(see also \cite{pardoux_BackwardStochasticDifferential_1992}) we know
that there exists a unique solution $\left(\bar X_\cdot^{t,x}, Y_\cdot^{t,u,x}, Z_\cdot^{t,u,x}\right)\in
\mathscr S^2\oplus\mathscr M^2\left([t,T];\R^{L_1}\times \R^{L_2}\right)$ of \eqref{t2-1-2}.

Then we define
$\T: C^{1,\infty}_b\left([0,T]\times \R^{L_1};\R^{L_2}\right)\to  C^{1,\infty}_b\left([0,T]\times \R^{L_1};\R^{L_2}\right)$
by $\T(u):=w$, where
$w(t,x):=Y_t^{t,u,x}$ with $\left(\bar X_\cdot^{t,x}, Y_\cdot^{t,u,x}, Z_\cdot^{t,u,x}\right)$ being
the unique solution of \eqref{t2-1-2}.
Since the coefficients of \eqref{t2-1-2} have bound derivatives with respect to space variables up to any order, it is standard to verify that
$w\in C^{1,\infty}_b\left([0,T]\times \R^{L_1};\R^{L_2}\right)$ (see e.g.
the arguments of \cite{{pardoux_BackwardStochasticDifferential_1992}} for details),
and according to
\cite[Theorem 3.2]{pardoux_BackwardStochasticDifferential_1992} we know that $w$ is non-random,
\begin{equation}\label{t2-1-4}
Y_s^{t,u,x}=w\left(s,\bar X_s^{t,x}\right),\ Z_s^{t,u,x}=\bar A\left(s,\bar X_s^{t,x}\right)\bar \nabla w\left(s,\bar X_s^{t,x}\right),
\end{equation}
and $w$ solves the following
quasi-linear parabolic system,
\begin{equation}\label{t2-1-3}
\begin{cases}
&\frac{\partial w(t,x)}{\partial t}=-\frac{1}{2}\bar{\mathcal{L}}_t w(t,\cdot)(x)+\frac{1}{2}\bar
\Gamma\left(w(t,x)\right)\left(\bar A(t,x)\bar \nabla u(t,x),
\bar A(t,x)\bar \nabla u(t,x)\right),\\
& w(T,x)=\bar h(x),\ \ \forall\ t\in [0,T],
\end{cases}
\end{equation}
where $\bar {\mathcal{L}}_t f(x):=\sum_{i=1}^{L_1}\bar A_i(t,\cdot)\left(\bar A_i(t,\cdot)f(\cdot)\right)(x)$
for every $f\in C^2(M;\R)$.
It is not difficult to verify that $\bar{\mathcal{L}}_s$ is the infinitesimal generator of
$\{\bar X_s^{t,x}\}_{s\in [t,T]}$.
In particular, we want to stress that although in \cite[Theorem 3.2]{pardoux_BackwardStochasticDifferential_1992},
the coefficients of forward equation for $\bar X_{\cdot}^{t,x}$ are time-independent, following the same arguments
the conclusion still holds for the case where the coefficients are time-dependent, see also \cite{WY}.

{\bf Step (ii)} For any $u\in C^{1,\infty}_b\left([0,T]\times \R^{L_1};\R^{L_2}\right)$, we could also view it as an
element in $C^{1,\infty}_b\left([0,T]\times M;\R^{L_2}\right)$ by restricting its definition in
$[0,T]\times M$. Let
\begin{equation*}
\norm{u}_{C^{0,1}([0,T]\times M)}:=\sup_{t\in [0,T];x\in M}\left|u(t,x)\right|_{\R^{L_2}}+\sup_{t\in [0,T];x\in M}
\left| \nabla^{g_t} u(t,x)\right|_{(T_x M,g_t)\times \R^{L_2}}.
\end{equation*}
Here $\left|\cdot\right|_{(T_x M,
g_t
)\times \R^{L_2}}$ denotes the norm for vectors in $T_x M \times \R^{L_2}$ with
respect to the product metric $g_t\times \bar g$, where $\bar g$ is the standard Euclidean metric on $\R^{L_2}$.
Let $w:=\T(u)$, by definition we have $w(t,x)=Y_t^{t,u,x}$, hence according to
\eqref{t2-1-2} it holds that
\begin{equation}\label{t2-1-5}
\begin{split}
&\quad\quad w(t,x)=\mathbb{E}\left[Y_t^{t,u,x}\right]\\
&=
\E\left[\bar h\left(\bar X_T^{t,x}\right)\right]-
\frac{1}{2}\int_t^T
\E\left[\bar \Gamma\left(Y_r^{t,u,x}\right)\left(\bar A(r,\bar X_r^{t,x})\bar \nabla
u(r,\bar X_r^{t,x}),\bar A(r,\bar X_r^{t,x})\bar \nabla
u(r,\bar X_r^{t,x})\right)\right]dr.
\end{split}
\end{equation}
As explained in Section \ref{section-2}, $\bar X_s^{t,x}=X_s^{t,x}\in M$, a.s. for
every $x\in M$ and $0\le t\le s \le T$, therefore we obtain that for every $x\in M$,
\begin{equation*}
\left|\bar h(\bar X_s^{t,x})\right|\le \sup_{x\in M}|\bar h(x)|=\sup_{x\in M}|h(x)|:=\norm{h}_{C^0(M)},\ a.s.,
\end{equation*}
\begin{equation*}
\begin{split}
\quad \left|\bar A(r,\bar X_r^{t,x})\bar \nabla
u(r,\bar X_r^{t,x})\right|
_{\R^{L_1L_2}}
&=\left|\nabla^{g_r}
u(r,\bar X_r^{t,x})\right|_{(T_{\bar X_r^{t,x}} M,g_r)\times \R^{L_2}}\\
&\le
\sup_{x\in M}\left| \nabla^{g_r} u(r,x)\right|_{(T_x M,g_r)\times \R^{L_2}},\ a.s..
\end{split}
\end{equation*}
Here the first equality is due to \eqref{t2-1-3a}.
So combining this with \eqref{t2-1-5} we obtain
\begin{equation}\label{t2-1-5a}
\begin{split}
\sup_{t\in [0,T];x\in M} |w(t,x)|\le \norm{h}_{C^0(M)}+c_1T\norm{u}_{C^{0,1}([0,T]\times M)}^2,
\end{split}
\end{equation}
where we also use the property that $\sup_{p\in \R^{L_2}}|\bar \Gamma(p)|
_{L(\R^{L_2}\times \R^{L_2};\R^{L_2})}<\infty$.

Since $\bar A_i(t,\cdot)\in C_b^\infty(\R^{L_1};\R^{L_1})$, it is well known that (see e.g. \cite{Li}) we can find
a $\Pp$-null set $\Lambda$ such that $\bar X_s^{t,\cdot}(\omega)\in C_b^1(\R^{L_1};\R^{L_1})$ for every $\omega \notin \Lambda$ and
$0\le t\le s \le T$. Moreover, let $\bar \nabla \bar X_s^{t,x}$ denote the gradient of $\bar X_s^{t,x}$ with respect to
 space variable $x\in \R^{L_1}$, then it satisfies the following linear SDE
\begin{equation*}
\begin{split}
\bar \nabla  \bar X_s^{t,x}=\mathbf{I}+\sum_{i=1}^{L_1}\int_t^s \bar \nabla \bar A_i\left(r,\bar X_r^{t,x}\right)\bar
\nabla \bar X_r^{t,x} \circ dW_r^i,\ \forall\ 0\le t \le s\le T,\ x\in \R^{L_1},
\end{split}
\end{equation*}
where $\mathbf{I}$ is the identity map from $\R^{L_1}$ to $\R^{L_1}$. Therefore based on this (linear) expression  and applying
standard arguments (see  e.g. \cite{pardoux_BackwardStochasticDifferential_1992} for details) it is not difficult to verify that
\begin{equation}\label{t2-1-6}
\sup_{0\le t\le s\le T;x\in \R^{L_1}}
\E\left[\left|\bar \nabla \bar X_s^{t,x}\right|\right]\le c_2e^{c_3T}
\end{equation}
for some positive constants $c_2,c_3$.

Meanwhile according to \eqref{e2-3} we can define the extension $\bar h:\R^{L_1}\to \R^{L_2}$ by
\begin{equation*}
\bar h(x)=\phi_M^0\left({\rm dist}^0_M(x)\right)h\left(P_M^0(x)\right),\ \forall\ x\in \R^{L_1},
\end{equation*}
where ${\rm dist}^0_M(x)$ denotes the distance between $x$ and $\Phi^0(M)$ in $\R^{L_1}$, $P_M^0(x)$ is the nearest
projection map from $B\left(\Phi^0(M),3\delta_{(M,g_0)}\right)$ (here we use the isometric embedding $\Phi^0:M \to \R^{L_1}$
associated with $g_0$) to $\R^{L_1}$,
$\phi_M^0$ is a cut-off function on $B\left(\Phi^0(M),3\delta_{(M,g_0)}\right)$ defined by the same way of \eqref{e2-3}. Based on
the expression above it is easy to see that
\begin{align*}
\sup_{x\in \R^{L_1}}\left|\bar \nabla \bar h(x)\right|\le c_4\left(\sup_{x\in M}|h(x)|+
\sup_{x\in M}|\nabla^{g_0}h(x)|\right):=c_4\norm{h}_{C^1(M;g_0)},
\end{align*}
combining which with \eqref{t2-1-6} yields immediately that
for every $t\in [0,T]$ and $x\in M\subset \R^{L_1}$,
\begin{equation}\label{t2-1-7}
\begin{split}
\left|\nabla^{g_t}\E\left[\bar h\left(\bar X_T^{t,\cdot}\right)\right](x)\right|&
\le c_5\left|\bar \nabla \E\left[\bar h\left(\bar X_T^{t,\cdot}\right)\right](x)\right|\\
&\le c_5\sup_{x\in \R^{L_1}}\left|\bar \nabla \bar h(x)\right|\cdot
\sup_{x\in \R^{L_1}}\E\left[\left|\bar \nabla \bar X_T^{t,x}\right|\right]\\
&\le c_6e^{c_7 T}\norm{h}_{C^1(M;g_0)}.
\end{split}
\end{equation}

 According to \eqref{e2-1a}, $\{X_{s}^{t,x}; 0\le t\le s \le T,x\in M\}$ is the solution to
SDE \eqref{e2-1}, which is an $M$-valued Markovian process with time-dependent infinitesimal generator
$\mathcal {L}_s=\Delta_{g_s}$, $s\in [t,T]$. So the distribution of $\{X_{s}^{t,x}; 0\le t\le s \le T,x\in M\}$ is
the same as that of $g_t$-Brownian motion defined by \cite{ACT,C}, hence due to \cite[Thereom 3.4]{Cheng}
there exist positive constants $c_8, c_9$ such that for all $f\in C_b(M)$ and $0\le t \le s \le T$,
\begin{equation}\label{t2-1-8}
\begin{split}
\sup_{x\in M}\left|\nabla^{g_t}\E\left[f(X_s^{t,\cdot})\right](x)\right|_{(T_x M,g_t)}&
\le \frac{c_8e^{c_9\mathcal{R}}}{\sqrt{(s-t)\wedge 1}}\sup_{x\in M}|f(x)|.
\end{split}
\end{equation}
Here $\mathcal{R}$ is the constant defined by \eqref{e2-0}. Therefore we obtain for every
$0\le t\le s\le T$,
\begin{align*}
&\quad \sup_{x\in M}\Big|\nabla^{g_t}\E\left[\bar \Gamma\left(Y_s^{t,u,\cdot}\right)
\left(\bar A\left(s,\bar X_{s}^{t,\cdot}\right)\bar \nabla u\left(s,\bar X_s^{t,\cdot}\right)
, \bar A\left(s,\bar X_{s}^{t,\cdot}\right)\bar \nabla u\left(s,\bar X_s^{t,\cdot}\right)\right)\right](x)\Big|\\
&=\sup_{x\in M}\Big|\nabla^{g_t}\E\left[\bar \Gamma\left(w\left(s,\bar X_s^{t,\cdot}\right)\right)
\left(\bar A\left(s,\bar X_{s}^{t,\cdot}\right)\bar \nabla u\left(s,\bar X_s^{t,\cdot}\right)
, \bar A\left(s,\bar X_{s}^{t,\cdot}\right)\bar \nabla u\left(s,\bar X_s^{t,\cdot}\right)\right)\right](x)\Big|\\
&=\sup_{x\in M}\Big|\nabla^{g_t}\E\left[\bar \Gamma\left(w\left(s, X_s^{t,\cdot}\right)\right)
\left(\nabla^{g_s} u\left(s, X_s^{t,\cdot}\right)
, \nabla^{g_s} u\left(s, X_s^{t,\cdot}\right)\right)\right](x)\Big|\\
&\le \frac{c_{10}e^{c_9\mathcal{R}}}{\sqrt{(s-t)\wedge 1}}\sup_{x\in M}
\Big|\left|\bar \Gamma(w(s,x))\right|
_{L(\R^{L_2}\times \R^{L_2};\R^{L_2})}
\cdot \left|\nabla^{g_s} u\left(s,x\right)\right|_{(T_xM,g_s)\times \R^{L_2}}^2\Big|
\le \frac{c_{11}\norm{u}_{C^{0,1}([0,T]\times M)}^2}{\sqrt{(s-t)\wedge 1}}.
\end{align*}
Here the first step follows from \eqref{t2-1-4}, the second step is due to \eqref{t2-1-3a} and the fact that
$\bar X_{s}^{t,x}=X_s^{t,x}$ a.s. for every $0\le t \le s \le T$ and $x\in M \subset \R^{L_1}$, in the third step we have applied
\eqref{t2-1-8}.  combining this with \eqref{t2-1-5a} and \eqref{t2-1-7} yields that
\begin{equation*}
\begin{split}
\norm{w}_{C^{0,1}([0,T]\times M)}&=\sup_{t\in [0,T];x\in M}|w(t,x)|+
\sup_{t\in [0,T];x\in M}\left|\nabla^{g_t}w(t,x)\right|_{(T_xM,g_t)
\times \R^{L_2}}\\
&\le c_{12}\left(e^{c_{13}T}\norm{h}_{C^1(M;g_0)}+T\norm{u}_{C^{0,1}([0,T]\times M)}^2+\sup_{t\in [0,T]}\int_{t}^T
\frac{\norm{u}_{C^{0,1}([0,T]\times M)}^2}{\sqrt{(s-t)\wedge 1}}ds \right)\\
&\le c_{14}\left(e^{c_{13}T}\norm{h}_{C^1(M;g_0)}+T\norm{u}_{C^{0,1}([0,T]\times M)}^2+
\sqrt{T}\norm{u}_{C^{0,1}([0,T]\times M)}^2\right),
\end{split}
\end{equation*}
which means
\begin{equation}\label{t2-1-9a}
\begin{split}
\norm{\T(u)}_{C^{0,1}([0,T]\times M)}&\le
c_{14}\left(e^{c_{13}T}\norm{h}_{C^1(M;g_0)}+(\sqrt{T}+T)\norm{u}_{C^{0,1}([0,T]\times M)}^2\right)\\
&\le c_{14}\left(e^{c_{13}T}\norm{h}_{C^{0,1}([0,T]\times M)}+(\sqrt{T}+T)\norm{u}_{C^{0,1}([0,T]\times M)}^2\right).
\end{split}
\end{equation}

{\bf Step (iii)} Given $u_1,u_2\in C_b^{1,\infty}([0,T]\times \R^{L_1};\R^{L_2})$,
set $w_1:=\T(u_1)$, $w_2:=\T(u_2)$. By \eqref{t2-1-3a}, \eqref{t2-1-4} and \eqref{t2-1-5} we obtain
that for every $x\in M\subset \R^{L_1}$,
\begin{equation}\label{t2-1-9}
\begin{split}
\quad w_2(t,x)-w_1(t,x)
&=
\frac{1}{2}\int_t^T \E\Big[\bar \Gamma\left(w_1(s,X_s^{t,x})\right)\left(\nabla^{g_s}u_1(s,X_s^{t,x}), \nabla^{g_s}u_1(s,X_s^{t,x})\right)\\
&\quad \quad -\bar \Gamma\left(w_2(s,X_s^{t,x})\right)\left(\nabla^{g_s}u_2(s,X_s^{t,x}), \nabla^{g_s}u_2(s,X_s^{t,x})\right)\Big]ds.
\end{split}
\end{equation}
Based on \eqref{t2-1-9} we conclude that
\begin{equation}\label{t2-1-11}
\begin{split}
& \quad\quad\sup_{x\in M}|w_1(t,x)-w_2(t,x)|\\
&\le \int_t^T \sup_{x\in M}\Big|\bar \Gamma\left(w_1(s,x)\right)\left(\nabla^{g_s}u_1(s,x), \nabla^{g_s}u_1(s,x)\right)-
\bar \Gamma\left(w_2(s,x)\right)\left(\nabla^{g_s}u_2(s,x), \nabla^{g_s}u_2(s,x)\right)\Big|ds\\
&\le \int_t^T\Big(\sup_{x\in M}\Big|\left(\bar \Gamma\left(w_1(s,x)\right)-
\bar \Gamma\left(w_2(s,x)\right)\right)\left(\nabla^{g_s}u_1(s,x), \nabla^{g_s}u_1(s,x)\right)\Big|\\
&+ \sup_{x\in M}\Big|\bar \Gamma\left(w_2(s,x)\right)\left(\nabla^{g_s}u_1(s,x), \nabla^{g_s}u_1(s,x)\right)-
\bar \Gamma\left(w_2(s,x)\right)\left(\nabla^{g_s}u_2(s,x), \nabla^{g_s}u_2(s,x)\right)\Big|\Big)ds\\
&\le c_{15}\Big(\norm{u_1}_{C^{0,1}([0,T]\times M)}^2\cdot\int_t^T \sup_{x\in M}|w_1(s,x)-w_2(s,x)|ds\\
&\quad \quad +T\left(\norm{u_1}_{C^{0,1}([0,T]\times M)}+\norm{u_2}_{C^{0,1}([0,T]\times M)}\right)\cdot \norm{u_1-u_2}_{C^{0,1}([0,T]\times M)}\Big).
\end{split}
\end{equation}
Therefore using Gronwall's lemma we arrive at
\begin{equation}\label{t2-1-10}
\begin{split}
\quad \sup_{t\in [0,T]\times M}|w_1(t,x)-w_2(t,x)|&\le  c_{16}T e^{c_{17}T\norm{u_1}_{C^{0,1}([0,T]\times M)}^2}\\
&\cdot \left(\norm{u_1}_{C^{0,1}([0,T]\times M)}+\norm{u_2}_{C^{0,1}([0,T]\times M)}\right)\norm{u_1-u_2}_{C^{0,1}([0,T]\times M)}.
\end{split}
\end{equation}

Meanwhile we can deduce from \eqref{t2-1-8}, \eqref{t2-1-9} and \eqref{t2-1-10} that
\begin{equation*}
\begin{split}
&\quad \quad\sup_{x\in M}|\nabla^{g_t}w_1(t,x)-\nabla^{g_t}w_2(t,x)|_{(T_x M,g_t)\times \R^{L_2}}\\
&\le \int_t^T \frac{c_{18}}{\sqrt{(s-t)\wedge 1}} \sup_{x\in M}\Big|\bar \Gamma\left(w_1(s,x)\right)\left(\nabla^{g_s}u_1(s,x), \nabla^{g_s}u_1(s,x)\right)\\
&\quad \quad \quad -\bar \Gamma\left(w_2(s,x)\right)\left(\nabla^{g_s}u_2(s,x), \nabla^{g_s}u_2(s,x)\right)\Big|ds\\
&\le c_{19}\Big(\norm{u_1}_{C^{0,1}([0,T]\times M)}^2\sup_{t\in [0,T]\times M}|w_1(t,x)-w_2(t,x)|\\
&\quad +\left(\norm{u_1}_{C^{0,1}([0,T]\times M)}+\norm{u_2}_{C^{0,1}([0,T]\times M)}\right)\cdot \norm{u_1-u_2}_{C^{0,1}([0,T]\times M)}
\Big)\cdot \int_t^T \frac{1}{\sqrt{(s-t)\wedge 1}}ds\\
&\le c_{20}\sqrt{T}\Big(\left(1+T\norm{u_1}_{C^{0,1}([0,T]\times M)}^2 e^{c_{17}T\norm{u_1}_{C^{0,1}([0,T]\times M)}^2}\right)\\
&\quad \quad \quad \quad \quad \cdot\left(\norm{u_1}_{C^{0,1}([0,T]\times M)}+\norm{u_2}_{C^{0,1}([0,T]\times M)}\right)\cdot \norm{u_1-u_2}_{C^{0,1}([0,T]\times M)}\Big).
\end{split}
\end{equation*}
Here in the second inequality above we have applied the same arguments for \eqref{t2-1-11}.  Hence combining all above estimates together yields
that
\begin{equation}\label{t2-1-12}
\begin{split}
&\quad\quad \norm{\T(u_1)-\T(u_2)}_{C^{0,1}([0,T]\times M)}=\norm{w_1-w_2}_{C^{0,1}([0,T]\times M)}\\
&\le c_{21}\left(T^{1/2}+\left(T+T^{3/2}\norm{u_1}_{C^{0,1}([0,T]\times M)}^2\right) e^{c_{17}T\norm{u_1}_{C^{0,1}([0,T]\times M)}^2
}\right)\\
&\quad\quad \cdot\left(\norm{u_1}_{C^{0,1}([0,T]\times M)}+\norm{u_2}_{C^{0,1}([0,T]\times M)}\right)\cdot \norm{u_1-u_2}_{C^{0,1}([0,T]\times M)}.
\end{split}
\end{equation}

{\bf Step (iv)} Fixing a positive constant $K>\max\{4c_{14},2\}\cdot\norm{h}_{C^{0,1}([0,T]\times M)}$, where
$c_{14}$ is the same constant in \eqref{t2-1-9a}. Let
\begin{equation*}
\mathbb{B}(T;K):=\{u\in C_b^{1,\infty}([0,T]\times M;\R^{L_2}); \norm{u}_{C^{0,1}([0,T]\times M)}\le K\}.
\end{equation*}
By \eqref{t2-1-9a} we know immediately that for $T_1:=\min\left\{\frac{\log 2}{c_{13}}, \left(\frac{1}{4c_{14}K}\right)^2, \frac{1}{4c_{14}K}
\right\}$ (here $c_{13}$, $c_{14}$ are the same constants in \eqref{t2-1-9a}), it holds that
for every $0<T\le T_1$,
\begin{equation}\label{t2-1-13}
\T\left(u\right)\in  \mathbb{B}(T;K),\ \ \forall\ u\in \mathbb{B}(T;K).
\end{equation}
Furthermore, set $T_2:=\min\left\{\left(\frac{1}{8c_{21}K}\right)^2,
\left(\frac{1}{16c_{21}e}\right)^{2/3}\frac{1}{K^2}, \frac{1}{c_{17}K^2}, \frac{1}{16c_{21}eK}\right\}$, where
$c_{17}$ and $c_{21}$ are the same constants in \eqref{t2-1-12}. Then we can deduce from \eqref{t2-1-12} that
for every $0<T\le T_2$,
\begin{equation}\label{t2-1-14}
\norm{\T\left(u_1\right)-\T\left(u_2\right)}_{C^{0,1}([0,T]\times M)}\le \frac{1}{2}
\norm{u_1-u_2}_{C^{0,1}([0,T]\times M)},\ \forall\ u_1,u_2\in \mathbb{B}(T;K).
\end{equation}

Let $T_0:=\min\{T_1,T_2\}$. Set $u_0(t,x)=\bar h(x)$ for all $t\in [0,T_0]$ and $x\in M$ and we define
$u_n\in C^{1,\infty}_b\left([0,T_0]\times M;\R^{L_2}\right)$
iteratively by
$u_{n+1}:=\T(u_n)$ for all $n\ge 0$. Note that $\norm{u_0}_{C^{0,1}([0,T_0]\times M)}$ $=\norm{h}_{C^{0,1}([0,T_0]\times M)}$
$\le K$,
so according to \eqref{t2-1-13} and \eqref{t2-1-14} we obtain that for all $n\ge 1$,
\begin{equation*}
\begin{split}
&u_n=\T(u_{n-1})\in \mathbb{B}(T_0;K),\\
&\norm{u_{n+1}-u_n}_{C^{0,1}([0,T_0]\times M)}=\norm{\T(u_n-u_{n-1})}_{C^{0,1}([0,T_0]\times M)}
\le \frac{1}{2}\norm{u_{n}-u_{n-1}}_{C^{0,1}([0,T_0]\times M)}.
\end{split}
\end{equation*}
Hence by Banach fixed point theorem, we could find a $\mathbf{u}\in C_b^{0,1}([0,T_0]\times M;\R^{L_2})$ such that
\begin{equation}\label{t2-1-15}
\lim_{n \to \infty}\norm{u_n-\mathbf{u}}_{C^{0,1}([0,T_0]\times M)}=0.
\end{equation}
For every $n\ge 1$, suppose $(\bar X_s^{t,x}, Y_s^{t,n,x},Z_s^{t,n,x})$ is the solution of \eqref{t2-1-2} with $u=u_{n-1}$,
by \eqref{t2-1-3a}, \eqref{t2-1-4} and definition of $u_n$ we have
\begin{equation}\label{t2-1-16}
Y_s^{t,n,x}=u_{ n}(s,X_s^{t,x}),\ Z_s^{t,n,x}=\nabla^{g_s}u_{n}(s,X_s^{t,x}),\ \ x\in M,\ 0\le t\le s \le T_0,
\end{equation}
where we have also used the fact that $\bar X_s^{t,x}=X_s^{t,x}\in M$ a.s. for all $x\in M\subset \R^{L_1}$.

Let
\begin{equation*}
Y_s^{t,x}:=\mathbf{u}(s,X_s^{t,x}),\ Z_s^{t,x}:=\nabla^{g_s}\mathbf{u}(s,X_s^{t,x}),\ 0\le t\le s\le T_0,\ x\in M.
\end{equation*}
With the expression \eqref{t2-1-16} for $Y_s^{t,n,x}$, $Z_s^{t,n,x}$, taking $n \to \infty$ in \eqref{t2-1-2} and applying
\eqref{t2-1-15} we deduce immediately that $(X_\cdot^{t,x}, Y_\cdot^{t,x},Z_\cdot^{t,x})
\in
\mathscr S^2\oplus\mathscr M^2\left([t,T_0];\R^{L_1}\times \R^{L_2}\right)$ is a solution of \eqref{Gamma-bsde}
which satisfies \eqref{t2-1-0a}. Moreover,
by Proposition \ref{p4-1}
we know that $Y_s^{t,x}\in N$ a.s. for all $0\le t\le s\le T_0$ and $x\in M$.

Suppose $(X_\cdot^{t,x}, \tilde Y_\cdot^{t,x},\tilde Z_\cdot^{t,x})
\in
\mathscr S^2\oplus\mathscr M^2\left([t,T_0];\R^{L_1}\times \R^{L_2}\right)$
(since the forward SDE has a unique solution so we take $X_\cdot^{t,x}$ here) is another solution of \eqref{Gamma-bsde} which satisfies
\eqref{t2-1-0a}. Repeating the arguments for the proof
(of uniqueness of the solution) in \cite[Theorem 3.1]{kupper_MultidimensionalMarkovianFBSDEs_2019}, we can show that
$\tilde Y_s^{t,x}=Y_s^{t,x}$ a.s. for every $0\le t\le s\le T$ and $\tilde Z_s^{t,x}(\omega)=Z_s^{t,x}(\omega)$ for
$ds\times \Pp$ a.s.
$(s,\omega)\in (t,T_0)\times \Omega$, which implies the uniqueness of the solution of \eqref{Gamma-bsde}.

According to \eqref{t2-1-3} and restricting it at $x\in M$ we know that $u_n$ satisfies the following equation
\begin{equation*}
\begin{cases}
&\frac{\partial u_n(t,x)}{\partial t}=-\frac{1}{2}\Delta_{g_t}u_n(t,x)+
\frac{1}{2}\bar \Gamma(u_n(t,x))\left(\nabla^{g_t}u_{n-1}(t,x),
\nabla^{g_t}u_{n-1}(t,x)\right),\\
& u_n(T_0,x)=h(x),\ \ t\in [0,T_0], x\in M.
\end{cases}
\end{equation*}
Here we have used the property \eqref{t2-1-3a} and that
\begin{equation*}
\sum_{i=1}^{L_1}\bar A_i(t,\cdot)\left(\bar A_i(t,\cdot)f(\cdot)\right)(x)=
\sum_{i=1}^{L_1}A_i(t,\cdot)\left(A_i(t,\cdot)f(\cdot)\right)(x)=\Delta_{g_t}f(x),\ \forall\ x\in M, f\in C_b^2(\R^{L_1}),
\end{equation*}
which is due to \eqref{e2-1a} and the fact that $\bar A_i(t,\cdot)$ is an extension of $A_i(t,\cdot)$.

Applying integration by parts formula we obtain that for every
$f\in C_b^{1}(M;\R)$, $0\le t\le T_0$,
\begin{equation*}
\begin{split}
& \quad \quad \int_M h(x)f(x)\,dx^{g_{T_0}}-\int_M u_n(t,x) f(x)\,dx^{g_t}
-\int_t^{T_0}\int_M u_n(s,x) f(x)\frac{\partial}{\partial s}\left(dx^{g_s}\right)\,ds\\
&=
\frac{1}{2}\int_t^{T_0}\int_M \left\langle \nabla^{g_s} u_n(s,x), \nabla^{g_s}f(x)\right\rangle
_{(T_x M,g_s)}
\,dx^{g_s}ds\\
&+\frac{1}{2}\int^{T_0}_t\int_M \bar \Gamma(u_n(s,x))\big(\nabla^{g_s}u_{n-1}(s,x),
\nabla^{g_s}u_{n-1}(s,x)\big) f(x)\,dx^{g_s} ds,
\end{split}
\end{equation*}
where $dx^{g_t}$ denotes the volume measure on $M$ induced by Riemannian metric $g_t$. Therefore letting
$n \to \infty$ in above and applying \eqref{t2-1-15} we know that $\mathbf{u}\in C^{0,1}_b([0,T_0]\times M; \R^{L_2})$ satisfies that
for every $t\in [0,T_0]$ and $f\in C_b^1(M;\R)$,
\begin{equation*}
  \begin{split}
  & \quad \quad \int_M h(x)f(x)\,dx^{g_{T_0}}-\int_M \mathbf{u}(t,x) f(x)\,dx^{g_t}
  -\int_t^{T_0}\int_M \mathbf{u}(s,x) f(x)\frac{\partial}{\partial s}\left(dx^{g_s}\right)\,ds\\
  &=
  \frac{1}{2}\int_t^{T_0}\int_M \left\langle \nabla^{g_s} \mathbf{u}(s,x), \nabla^{g_s}f(x)\right\rangle_{(T_x M,g_s)}\,dx^{g_s}ds\\
  &+\frac{1}{2}\int^{T_0}_t\int_M \bar \Gamma(\mathbf{u}(s,x))\big(\nabla^{g_s}\mathbf{u}(s,x),
  \nabla^{g_s}\mathbf{u}(s,x)\big) f(x)\,dx^{g_s} ds,
  \end{split}
  \end{equation*}
which implies $\mathbf{u}\in C^{0,1}_b([0,T_0]\times M; \R^{L_2})$ is a distributional solution to  the following quasi-linear
parabolic system,
\begin{equation}\label{t2-1-17}
\begin{split}
&\frac{\partial u(t,x)}{\partial t}=-\frac{1}{2}\Delta_{g_t} u(t,x)+
\frac{1}{2}\bar \Gamma(u(t,x))\left(\nabla^{g_t}u(t,x),
\nabla^{g_t}u(t,x)\right),\\
&u(T_0,x)=h(x),\ \ t\in [0,T_0], x\in M.
\end{split}
\end{equation}
Based on prior regularity condition $\sup_{(t,x)\in[0,T_0]\times M}|\nabla^{g_t} \mathbf{u}(t,x)|_{(T_xM;g_t)\times \R^{L_2}}<\infty$
and $h\in C_b^\infty(M;N)$,
according to standard theory of quasi-linear parabolic system (see e.g. \cite[Chapter V and VII]{LSU},
or \cite[Appendix A]{M}) we know that $\mathbf{u}\in C_b^{1,\infty}([0,T_0]\times M;\R^{L_2})$ and $\mathbf{u}(t,x)=Y_t^{t,x}$ is a
classical solution of \eqref{t2-1-17}.

Note that $\mathbf{u}(t,x)=Y_t^{t,x}\in N$, so by \eqref{e2-5} we have
\begin{equation*}
\begin{split}
&\quad-\Delta_{g_t} \mathbf{u}(t,x)+
\bar \Gamma(\mathbf{u}(t,x))\left(\nabla^{g_t}\mathbf{u}(t,x),
\nabla^{g_t}\mathbf{u}(t,x)\right)\\
&=-\Delta_{g_t} \mathbf{u}(t,x)+
\Gamma(\mathbf{u}(t,x))\left(\nabla^{g_t}\mathbf{u}(t,x),
\nabla^{g_t}\mathbf{u}(t,x)\right)=-\Delta_{g_t,\mathbf h}\mathbf{u}(t,x),\ \forall\ x\in M.
\end{split}
\end{equation*}
combining this with \eqref{t2-1-17} we know $\mathbf{u}(t,x)=Y_t^{t,x}$ is the unique solution of \eqref{t2-1-1}
in $C_b^{1,\infty}([0,T_0]\times M;N)$.
\end{proof}

\begin{section}{Appendix}\label{section-4}
The following proposition has been shown in \cite[Theorem 3.2]{chen_StudyBackwardStochastic_2020}. For
convenience of readers, we will also give a proof here.
\begin{proposition}\label{p4-1}
Suppose that $\left(X_{\cdot}^{t,x},Y_{\cdot}^{t,x},Z_{\cdot}^{t,x}\right)\in
\mathscr S^2\oplus\mathscr M^2\left([t,T_0];\R^{L_1}\times \R^{L_2}\right)$
is a solution
of \eqref{Gamma-bsde} in time
interval $0\le t \le s \le T_0$ which satisfies \eqref{t2-1-0a}, then we have
$Y_s^{t,x}\in N$ a.s. for every $0\le t\le s\le T_0$ and $x\in M$.
\end{proposition}
\begin{proof}
Let $\delta_N>0$ be the constant introduced in Section \ref{section-2} such that
the nearest projection map $P_N:B(N,3\delta_N)\to N$ and the square of distance function
${\rm dist}^2_N: B(N,3\delta_N)\to \R_+$ are smooth.  Choosing a truncation function
$\chi\in C_b^\infty(\R)$ satisfying that $\chi'\ge 0$ and
\begin{equation*}
\chi(s)=
\begin{cases}
& s,\ \ \ \ \ \ \ s\le \delta_N^2,\\
& 4\delta_N^2,\ \ \ \ s>4\delta_N^2.
\end{cases}
\end{equation*}
We define $G:\R^{L_2} \to \R_+$ by $G(p):=\chi\Big({\rm dist}^2_N(p)\Big),\ p\in \R^{L_2}$.
It is easy to verify that $G(p)=4\delta_N^2$ for every $p\in \R^{L_2}$ with
${\rm dist}_N(p)>2\delta_N$.

Since $G(p)={\rm dist}_N^2(p)=|p-P_N(p)|^2$ when $p\in B(N,\delta_N)$,
it holds that for every $p\in B(N,\delta_N)$, $u=(u_1,\cdots, u_{L_2})\in \R^{L_2}$,
\begin{equation*}
\begin{split}
\bar \nabla^2 G(p)(u,u)=&2\sum_{k=1}^{L_2}\left(\sum_{i=1}^{L_2} u_i\left(\delta_{ik}-\frac{\partial P_N^k}{\partial p_i}(p)\right)\right)^2 -
2\sum_{i,j,k=1}^{L_2}(p_k-P_N^k(p))\frac{\partial^2 P_N^k}{\partial p_i \partial p_j}(p)u_iu_j\\
&\ge -2\sum_{i,j,k=1}^{L_2} (p_k-P_N^k(p))\frac{\partial^2 P_N^k}{\partial p_i \partial p_j}(p)u_iu_j,
\end{split}
\end{equation*}
where $\delta_{ij}$ denotes the Kronecker delta function
(i.e. $\delta_{ij}=0$ if $i\neq j$ and $\delta_{ij}=1$ when $i=j$), $P_N^k(p)$ denotes the $k$-th components of $P_N(p)$
(i.e. $P_N(p)=(P_N^1(p),\cdots P_N^{L_2}(p))$).
By the definition \eqref{e2-3} of $\bar \Gamma$ we can obtain that
for every $p\in B(N,\delta_N)$, $u=(u_1,\cdots, u_{L_2})\in \R^{L_2}$,
\begin{equation*}
\begin{split}
\left\langle \bar \nabla G(p), \bar \Gamma(p)(u,u)\right\rangle_{\R^{L_2}}&=
2\sum_{i,j,k=1}^{L_2} (p_k-P_N^k(p))\frac{\partial^2 P_N^k}{\partial p_i \partial p_j}(P_N(p))u_iu_j\\
&-
2\sum_{i,j,k,l=1}^{L_2} (p_k-P_N^k(p))\frac{\partial P_N^k}{\partial p_l}(p)\frac{\partial^2 P_N^l}{\partial p_i \partial p_j}(P_N(p))u_iu_j\\
&=2\sum_{i,j,k=1}^{L_2} (p_k-P_N^k(p))\frac{\partial^2 P_N^k}{\partial p_i \partial p_j}(P_N(p))u_iu_j.
\end{split}
\end{equation*}
Here  the last step follows from the fact that
$$\sum_{k=1}^{L_2}(p_k-P_N^k(p))\frac{\partial P_N^k}{\partial p_l}(p)=0,$$
which is due to the property that $\frac{\partial P_N}{\partial p_l}(p)\in T_p N$ and $p-P_N(p)\in T_p^{\bot} N$.
Putting all above estimates together we have
\begin{equation*}
\begin{split}
&\quad\ \bar \nabla^2 G(p)(u,u)+\left\langle \bar \nabla G(p), \bar \Gamma(p)(u,u)\right\rangle_{\R^{L_2}}\\
&\ge  2\sum_{i,j,k=1}^{L_2} (p_k-P_N^k(p))\left(\frac{\partial^2 P_N^k}{\partial p_i \partial p_j}(P_N(p))
-\frac{\partial^2 P_N^k}{\partial p_i \partial p_j}(p)\right)u_iu_j\\
&\ge -c_2{\rm dist}_N^2(p)|u|^2=-c_2G(p)|u|^2,\ p\in B(N,\delta_N),\ u=(u_1,\cdots, u_L)\in \R^{L_2}.
\end{split}
\end{equation*}
Still by the definition of $G$ and $\bar \Gamma$ we know that for every $p\in \R^{L_2}/B(N,\delta_N)$ and $u\in \R^{L_2}$,
\begin{align*}
&\bar \nabla^2 G(p)(u,u)+\left\langle \bar \nabla G(p), \bar \Gamma(p)(u,u)\right\rangle_{\R^{L_2}}
\ge -c_3(1+|u|^2)\ge -c_4G(p)(1+|u|^2),
\end{align*}
where the second inequality is due to the fact that $G(p)\ge \delta_N^2$ for every $p\in \R^{L_2}/B(N,\delta_N)$.

combining above  two estimates yields that
\begin{equation}\label{t2-2-4}
\bar \nabla^2 G(p)(u,u)+\left\langle \bar \nabla G(p), \bar \Gamma(p)(u,u)\right\rangle_{\R^{L_2}}
\ge -c_5G(p)(1+|u|^2),\ \forall\ p,u\in \R^{L_2}.
\end{equation}
Hence applying It\^o's formula to the second equation in \eqref{Gamma-bsde} we obtain that
for every $0\le t\le s\le T_0$ and $x\in M$,
\begin{align*}
0=G\left(\bar h\left(X_{T_0}^{t,x}\right)\right)&=G(Y_s^{t,x})+\sum_{i=1}^{L_1}\int_s^{T_0}\langle
\bar \nabla G(Y_r^{t,x}), Z_r^{t,x,i} \rangle_{\R^{L_2}} dW_r^i\\
&+
\sum_{i=1}^{L_1}\int_s^{T_0}\frac{1}{2}\Big(\bar \nabla^2 G(Y_r^{t,x})(Z_r^{t,x,i},Z_r^{t,x,i})
+\left\langle \bar \nabla G(Y_r^{t,x}), \bar \Gamma(Y_r^{t,x})\left(Z_r^{t,x,i},Z_r^{t,x,i}\right)\right\rangle_{\R^{L_2}}\Big)dr\\
&\ge G(Y_s^{t,x})+ \sum_{i=1}^{L_1}\int_s^{T_0}\langle \bar \nabla G(Y_r^{t,x}), Z_r^{t,x,i} \rangle_{\R^{L_2}} dW_r^i-
\frac{c_5}{2}\int_s^{T_0}G(Y_r^{t,x})(1+|Z_r^{t,x}|^2)dr\\
&\ge G(Y_s^{t,x})+ \sum_{i=1}^{L_1}\int_s^{T_0}\langle \bar \nabla G(Y_r^{t,x}), Z_r^{t,x,i}\rangle_{\R^{L_2}} dW_r^i
-c_6\int_s^{T_0}G(Y_r^{t,x})dr,
\end{align*}
where the first step is due to the fact that $G\left(\bar h\left(X_{T_0}^{t,x}\right)\right)=0$ a.s.
for every $x\in M$ (since  $\bar{h}(X_{T_0}^{t,x}) \in N$ a.s.), we have used \eqref{t2-2-4} in the third step, and
the last step follows from \eqref{t2-1-0a}.

Hence taking the expectation in above inequality we arrive at
\begin{align*}
\E[G(Y_s^{t,x})]\le c_6\int_s^{T_0}\E[G(Y_r^{t,x})]dr,\ \forall\ s\in [t,T_0].
\end{align*}
So by  Gronwall's inequality we obtain $\E[G(Y_s^{t,x})]=0$ which implies $G(Y_s^{t,x})=0$ and $Y_s^{t,x}\in N$ a.s. for
every $s\in [t,T_0]$ and $x\in M$.
\end{proof}
\end{section}

\noindent \textbf{Acknowledgements.}
The research of Xin Chen and Wenjie Ye is supported by the National Natural Science Foundation of China (No.\ 11871338).
We thank the referees for helpful comments on an earlier version of this paper.

\end{document}